\documentclass[a4paper,11pt]{article}
\usepackage[latin1]{inputenc}
\usepackage[T1]{fontenc}
\usepackage{amsfonts}
\usepackage{amsmath}
\usepackage{enumerate}
\usepackage{amssymb}
\usepackage{hyperref}
\usepackage{amscd}
\usepackage{amsthm}
\usepackage[left=3cm,right=3cm,top=4cm,bottom=4cm]{geometry}
\usepackage{fontenc}
\usepackage{cancel}
\usepackage[pdftex]{graphicx}
\usepackage{xcolor}
\usepackage{mathrsfs}
\usepackage{color}

\newcommand{\ls}{\leqslant}
\newcommand{\gs}{\geqslant}
\renewcommand{\leq}{\leqslant}
\renewcommand{\leq}{\leqslant}

\newcommand{\R}{\mathbb R}

\newcommand{\scal}[2]{\left \langle #1 \, , \, #2 \right \rangle}
\renewcommand{\div}{\operatorname{div}}

\newcommand{\kg}{\kappa_\varphi}
\renewcommand{\k}{\kappa}

\newcommand{\eps}{\varepsilon}
\newtheorem{theo}{Theorem}
\newtheorem{defi}{Definition}
\newtheorem{prop}{Proposition}

\newtheorem{lem}{Lemma}
\newtheorem{rem}{Remark}
\title{Anisotropic curvature flow of immersed curves}
\date{}

\author{
Gwenael Mercier\footnote{RICAM, 
Johann Radon Institute for Computational and Applied Mathematics, 
Linz, Austria, email: gwenael.mercier@ricam.oeaw.ac.at},
Matteo Novaga\footnote{Dipartimento di Matematica, 
Universit\`a di Pisa, Pisa, Italy, 
e-mail: matteo.novaga@unipi.it},
Paola Pozzi\footnote{Universit\"at Duisburg-Essen,
Fakult\"at f\"ur Mathematik, Essen, Germany,
email: paola.pozzi@uni-due.de}
}

\begin{document}
\maketitle

\begin{abstract}
\noindent We prove short-time existence of $\varphi$-regular solutions to the anisotropic and crystalline curvature flow of immersed planar curves.

\end{abstract}

\tableofcontents

\section{Introduction}

In this paper we consider the anisotropic curvature motion for planar immersed curves. 
More precisely, given an initial curve $u_0 : S^1 \to \R^2$ and a norm $\varphi$ on $\R^2$,
which we call {\it anisotropy}, 
we show existence of a family $u(t,x)$ of curves satisfying 
the evolution equation
\begin{equation}\label{eqq}
u_{t}^{\perp} = \varphi^\circ(\nu) \k_\varphi \nu,
\end{equation}
where $u_{t}^{\perp}$ is the normal component of the velocity $u_t$,
$\nu$ is the Euclidean normal vector to the support of $u$, 
$\varphi^o$ is the dual norm of $\varphi$,
and $\k_\varphi$ denotes the anisotropic curvature (see Section \ref{sec1}). 
In particular, we prove a short time existence result 
which holds for any anisotropy $\varphi$, 
under a natural regularity assumption on the initial datum.

After the seminal paper by Brakke \cite{brakke78}, the
isotropic mean curvature flow has been intensively studied in the past years
(see for instance \cite{huisken84,ecker89,ecker91,mantegazza} and references therein) 
and the long-time behavior of the flow as well as its singularities are relatively well-understood. 

The first occurrence of anisotropic curvature flow appeared in \cite{angenent90},
where the author shows a short time existence result 
in two dimensions, for smooth and strictly convex anisotropies.
A few years later, in \cite{angenent91,oaks} it was shown that, at the maximal existence time the evolving curve either looses a self-intersection or shrinks to a point.  
In particular, if the initial curve is embedded, then it stays embedded, 
eventually becomes convex and then shrinks to homothetically a point (see \cite{gage94,cz,cz2}).

In the crystalline case, that is, when the anisotropy $\varphi$ is piecewise linear,
equation \eqref{eqq} reduces to a system of ODE's and 
existence and uniqueness of solutions was proved in \cite{GiGu:96}.
Shortly after, in \cite{St}, Stancu proved that a convex curve remains convex and 
shrinks homothetically to a point, while its shape approaches the boundary of the unit ball of $\varphi$,
and in \cite{GG} the authors showed that  an embedded initial curve
becomes eventually convex and then shrinks to a point by the result in \cite{St} (for a classification of the asymptotic behaviour of a polygonal curve when the time approaches the shrinking time see also \cite{Andrews}).

On the other hand few results are available for general anisotropies.
In the embedded case, an existence result analogous to the one in \cite{angenent90}
was established in 
\cite{chambolle13}, by means of an  implicit variational scheme introduced in \cite{ATW} 
in order to define a global weak solution to the anisotropic mean curvature flow.
The consistency of such solution with the solution constructed in \cite{GiGu:96} was
shown in \cite{AT}. 
An important notion in \cite{chambolle13} is that of $\varphi$-regular
flow (see \cite{BeNoPa,BeNo:99} and Definition \ref{defflow} below), which extends the notion of smooth evolution to general anisotropies. 
For such flows a uniqueness result was established in \cite{BeNo:99} (see also \cite{mcmf,GOS}).

Let us mention that very recently an existence and uniqueness result of global weak solutions to 
the anisotropic mean curvature flow in the embedded case, which holds in any dimensions and for general anisotropies, has been proved in \cite{CMP} (see also \cite{GGarch,GP} for a similar result in the crystalline case, in two and three dimensions).

The main purpose of this paper is to extend the existence result in \cite{chambolle13} to immersed curves. In order to do this, we regularize both the anisotropy and the initial curve, so that we can apply the result in \cite{angenent90} and obtain a smooth solution for short time.
Then, we show that we can pass to the limit in the regularization parameter, to obtain an existence result in the general case. We point out that, since the result in  \cite{BeNo:99} only applies to 
embedded solutions, at the moment we do not have a general uniqueness result.

\smallskip 

The plan of the paper is the following: in Section \ref{sec1},
we introduce the notation which  we use in the paper and define  the anisotropic curvature flow for immersed curves. In Section \ref{sec3}, we show that a curve with bounded $\varphi$-curvature
can be approximated by curves with bounded $\varphi_\eps$-curvature, where $\varphi_\eps$'s
are smooth anisotropies converging to $\varphi$ as $\eps\to 0$.
In Section \ref{asecsmooth}, we study the evolution of relevant geometric quantities under the flow, and prove that the curvature must blow up at the first singular time, as it happens in the isotropic case. This provides a uniform bound of the existence time of the approximate flows.
Finally, in Section \ref{secgen} we pass to the limit in the approximate flows and obtain a solution to the original anisotropic curvature motion in an interval $[0,T)$ where $T$ depends only on the initial data.


\section{Notation and preliminary definitions}\label{sec1}

We consider closed planar curves parametrized by $u:S^1 \times [0,T] \to \R^2$, where $S^1$ is identified in the obvious way with $[0, 2\pi]$.
We denote by $s$ the arc-length parameter of the curve (thus $\partial_s (\cdot)=\partial_x (\cdot)/ |u_x|$), by $\tau =u_x/|u_x|=u_s=(\sin \theta, -\cos \theta)$ its unit tangent and $\nu =(\cos \theta, \sin \theta)$ its unit normal.
The Euclidean scalar product in $\R^2$ is denoted by $\cdot$. The symbol $\perp $  stands for anti-clockwise rotation by $\pi/2$, therefore $(a,b)^\perp= (-b, a)$.
 Recall the classical Frenet formulas
\begin{align}
u_{ss}=\tau_s= \vec{\kappa}= \kappa \nu, \qquad \nu_s =-\kappa \tau.
\end{align}
Moreover recall that from the expression for $\nu_s$ one infers that for the scalar curvature $\kappa$ we have
\begin{align}
\kappa=\theta_s.
\end{align}

\subsection{Anisotropies}

\begin{defi}
 We call  \emph{anisotropy} a norm $\varphi: \R^2 \to [0, \infty)$.
We say that $\varphi$  is \emph{smooth}  if $\varphi \in C^{2}(\R^{2} \setminus \{ 0 \})$ and $\varphi$ is  \emph{elliptic} if $\varphi^{2}$ is strongly convex, that is, there exists $C > 0$ such that
\begin{align}\label{ell-cond}
 D^2(\varphi^{2}) \gs C \, \mbox{Id}
 \end{align}
in the distributional sense. 
\end{defi}
\begin{defi}
The set $W_{\varphi}:=\{\varphi \leq 1 \}$ is called  \emph{Wulff shape}. We say that
$\varphi$ is \emph{crystalline} if  $W_{\varphi}$ is a polygon.
\end{defi}
\begin{defi}
 Given an anisotropy $\varphi$, we introduce the \emph{polar norm} $\varphi^\circ$ relative to $\varphi$
 $$\varphi^\circ (x) = \sup\{ \xi \cdot x \, \vert \, \varphi(\xi) \ls 1\}.$$
\end{defi}

\begin{rem}\rm
Note that $\varphi$ is smooth and elliptic if and only if 
$\varphi^\circ$ is smooth and elliptic (\cite[\S~2]{chambolle13}). 

The ellipticity condition implies  that the Wulff shape is uniformly convex. 
Moreover, from \eqref{ell-cond} we infer that 
\begin{align}\label{ell-cond2}
D^{2}\varphi (\nu) \tau  \cdot \tau \geq \widetilde{C}, \qquad \widetilde{C}:=\frac{C}{2 \max  \{\varphi(\tilde{\nu}) \, | \, \tilde{\nu} \in S^{1}  \}},
\end{align}
for  unit vectors $\nu$ and $\tau$ with $\nu \cdot \tau =0$.
Indeed, condition \eqref{ell-cond} implies that $$2 \varphi(\nu) D^{2}\varphi(\nu) \xi \cdot \xi + 2 (D\varphi(\nu) \cdot \xi)^{2} \geq C|\xi|^{2}$$ for any $\xi \in \R^{2}$.  Given  $\tilde{\tau}$, $|\tilde{\tau}|=1$ with $D\varphi(\nu) \cdot \tilde{\tau} =0$, then we can write $\tilde{\tau}= \alpha \tau + \beta \nu$ for some $\alpha, \beta \in [-1,1]$, $\alpha \neq 0$. Then, using the homogeneity properties of $\varphi$ (namely $D\varphi(\xi) \cdot \xi = \varphi(\xi)$, $D^{2} \varphi(\xi) \xi =0$ for $\xi \in \R^{2}$, $\xi \neq 0$), we get
$$ C \leq  2\varphi(\nu) D^{2}\varphi(\nu) \tilde{\tau} \cdot \tilde{\tau}= 2\varphi(\nu)\alpha^{2}D^{2} \varphi(\nu) \tau \cdot \tau \leq  2\varphi(\nu) D^{2} \varphi(\nu) \tau \cdot \tau , $$
from which \eqref{ell-cond2} follows.
 Inequality \eqref{ell-cond2} will be used in Section~\ref{asecsmooth}.
\end{rem}

\subsection{$\varphi$-regular curves and the $RW_{\varphi}$-condition}

Following \cite[\S~2]{chambolle13} we give the following definition.

\begin{defi}\label{def4}
Let R>0. We say that a set $E\subset \R^2$ with non-empty interior satisfies the inner (resp. outer) $RW_{\varphi}$-condition if 
for any $x \in \partial E$, there exists  $y \in \R^{2}$ such that

$$ RW_{\varphi} +y \subseteq \overline{E} \quad (resp.\ RW_{\varphi} +y \subseteq \overline{E^{c}}), 
\qquad \text{ and } \qquad x \in \partial (RW_{\varphi} +y). $$
We also ask that for every $r < R$ and $x \in \R^{2}$, the set $(x+r W_\varphi) \cap E^c$ (resp. $(x+rW_\varphi) \cap E$) is connected. We say that a set $E$ satisfies the $RW_{\varphi}$-condition if
it satisfies both the inner and the outer $RW_{\varphi}$-condition.
\end{defi}
\begin{rem}\rm As observed in \cite[Remark~1]{chambolle13} (and using the connectedness assumption), if $E$ satisfies the $RW_{\varphi}$-condition for some $R>0$, then $\partial E$ is locally a Lipschitz graph. In particular a pathological set as depicted in \cite[Fig. 1]{chambolle13} cannot occur. Moreover if in addition $\varphi^\circ$ and $\varphi$ are smooth, then $\partial E$ is of class $C^{1,1}$ and $|\kappa_{\varphi}| \leq 1/R$ a.e. on $\partial E$.
\end{rem}

Since we will be working with immersed curves we need a localized version of the $RW_{\varphi}$-condition. 

\begin{defi}\label{deflocal}
Let $u$ be a closed curve which is locally a Lipschitz graph, that is,
for any $x\in S^{1}$ the image of $u$ coincides with the graph of a function $f_{x}$
in a neighborhood of $u(x)$.
We say that $u$ satisfies locally the inner (reps. outer) $RW_{\varphi}$-condition if
for any $x \in S^{1}$, there exist $y \in \R^{2}$ and $\rho>0$ such that
$ u(x) \in \partial (RW_{\varphi} +y)$ and 
$$ (RW_{\varphi} +y)\cap B_{\rho}(u(x)) \subseteq S^{-} 
\qquad (resp.\ (RW_{\varphi} +y)\cap B_{\rho}(u(x)) \subseteq S^{+}),$$
where $S^{-}$ and $S^{+}$ denote respectively the subgraph and the supergraph of 
the function $f_{x}$. 
We say that $u$ satisfies locally the $RW_{\varphi}$-condition if
it locally satisfies both the inner and the outer $RW_{\varphi}$-condition.
\end{defi}

Note that the above definition is weaker than Definition~\ref{def4}, in the sense that in the crystalline case information about the  curvature  gets lost. For instance if  $W_{\varphi} = \{ x \in \R^{2}, \| x \|_{L^{\infty}} \leq 1\}$ is a unit square, then  the set $E=\frac{1}{2} W_{\varphi}$ satisfies locally the $RW_{\varphi}$-condition for any positive $R$. 
This ``inconsistency'' is due to the fact that in the crystalline case the definition of  curvature (see Definition~\ref{def3} below) is no longer a  local concept. To retain information about the curvature, the geometrical arguments used  in Section~\ref{sec3} will use additional knowledge on the length of the flat sides.


As initial data for the anisotropic curve shortening flow we will consider immersed curves that 
admit a sufficiently regular anisotropic normal vector field (and hence locally satisfies a $RW_{\varphi}$-condition, by Lemma~\ref{lemmauno} below). More precisely, we give the following definition.

\begin{defi}\label{def3}
Let $u$ be a closed curve in $\R^2$ and let $s \in [0,L]$ denote its arc-length parameter. 
We say that $u$ is \emph{$\varphi$-regular} if the image of $u$ is 
locally a Lipschitz graph, and
there exists a Lipschitz  vector field  $n:[0,L] \to \R^{2}$, 
usually called Cahn-Hoffman vector field (see \cite{Ta:78}),
such that
$$ \quad n(s) \in \partial \varphi^\circ (\nu(s)) \text{ a.e. in }  [0,L],$$
where $\partial \varphi^\circ$ denotes the sub-differential of the anisotropy $\varphi^\circ$. 
We define the (scalar) \emph{anisotropic curvature} of $u$ by
$$ \kappa_\varphi  :=- N_{s} \cdot \tau,$$
where N is the Cahn-Hoffman vector field that minimizes the energy
$$\mathcal{K}(n)= \int_0^{L} (n_{s} \cdot \tau)^{2} \varphi^{\circ}(\nu) ds.$$
\end{defi}

Notice  that since a $\varphi$-regular curve is locally a Lipschitz graph its anisotropic curvature $\kappa_{\varphi}$ is well-defined almost everywhere in $s$. On the other hand different choices of Cahn-Hoffman vector fields are in general possible.
In the following we will always write $N$ when referring to the particular choice of Cahn-Hoffman vector field that realizes the minimum for the functional $\mathcal{K}$.
Finally recall that $  \kappa_\varphi  =- N_{s} \cdot \tau =-\div N$ and that the Euclidean tangential divergence coincides with the divergence of an extension for $N$ (\cite[Lemma~4.5]{bellettini01}).

\begin{rem}\label{rem3}\rm
If $u$ is smooth and $\varphi$ is smooth and elliptic, then $\kappa_{\varphi} =(D^{2} \varphi^\circ (\nu) \tau \cdot \tau) \kappa$ and the anisotropic curvature is controlled (from above and below) by the classical curvature $\kappa$, thanks to \eqref{ell-cond2}. 

If $u$ is piecewise linear and $\varphi$ is crystalline, then $\kappa_{\varphi}$ is no longer a local quantity. 
Indeed the curvature  of an edge $F \subset u(S^{1})$ of length $l$, parallel to an edge
 of the Wulff shape, is given by the expression
\begin{align*}
\kappa_{\varphi}^{F}=\delta_{F} \frac{l_{W_{\varphi}}}{l}
\end{align*}
where $l_{W_{\varphi}}$ is the length of the corresponding edge of the Wulff shape, and $\delta_{F} \in \{ 0, \pm1 \}$ is a local convexity factor. In particular,
if $0<|\kappa_{\varphi}| \leq C$ then $ l \geq l_{W_{\varphi}}/C$, that is,
$l$ must be long at least as the corresponding edge of the rescaled Wulff-shape $\frac{1}{C} W_{\varphi}$.
\end{rem}

Below we will use the following facts.
\begin{lem}\label{lemmauno}
Let $u$ be a $\varphi$-regular curve and let
$$F(s,d) := u(s) + d n(s),  \qquad n(s,d):=n(s)\qquad s \in [0, L].$$
Then, the curve $F(\cdot,d)$ is $\varphi$-regular with 
Cahn-Hoffman vector field $n(\cdot,d)$
for all $|d|<\Vert n_s \Vert_\infty$.
Moreover $u$ satisfies locally the $RW_\varphi$-condition with $R=\frac 1{\Vert n_s \Vert_\infty}$. 

For the particular choice of $n(s)=N(s)$ we have in addition that  $n(s,d)=N(s,d)$ and $\Vert N_s \Vert_\infty= \| \kappa_{\varphi} \|_{\infty}$.
\end{lem}

\begin{proof}
First of all notice that
$n_{s}$ is a tangential vector and therefore $n_{s}= (n_{s} \cdot \tau) \tau$ (in particular for the Cahn-Hoffman vector field $N$ we infer that 
$N_{s}= (N_{s} \cdot \tau) \tau =-\kappa_{\varphi} \tau$ and hence $\Vert N_s \Vert_\infty= \| \kappa_{\varphi} \|_{\infty}$).
The fact that $n_{s}$ is always parallel to $\tau$ is a consequence of $\varphi^2 (n) = 1$. Indeed, this implies that $\partial_s (\varphi^2 (n(s))) = 0$. 
Note that thanks to \cite[Th. 2.3.10]{clarke90}, 
if $\partial_s n \neq 0$, then $\partial_s(\varphi^2 (n(s))) =  \partial (\varphi^2)(n(s))\cdot\partial_s n(s)$, where $\partial$ denotes the subdifferential. 
Recalling that 
$$
\frac{\nu(s)}{\varphi^{\circ} (\nu(s))}\in
\frac{1}{2} \partial (\varphi^2)(n(s)),
$$ 
we then get
$$ \nu \cdot \partial_s n = 0,$$
this equality holding also when $\partial_s n = 0$.

The thesis now follows by the argument in  \cite[Lemma 3.3 and Lemma 3.4]{bellettini01}. Indeed, we first prove that $F$ is locally bi-Lipschitz from $[0, L] \times (-1/C, 1/C)$ onto its image in $\R^{2}$, where $C= \| n_{s} \|_{\infty}$.
To this aim, we follow \cite[Lemma 3.3]{bellettini01} and apply the inverse function theorem in a Lipschitz framework (cf. \cite[Th. 7.1.1]{clarke90}). We have to show that, if $F$ is differentiable at $(s_n,d_n)$,  $(s_n,d_n) \to (s,d)$, and there is a limit of $DF (s_n,d_n)$, then the limit  is nonsingular. 
To see this, we compute
$$DF(s_n,d_n) = \left( \tau(s_n) + d_n \partial_s n(s_n), n(s_n) \right)= \left( [1+d_n (n_{s}(s_{n}) \cdot \tau (s_{n})) ]\tau(s_n) , n(s_n) \right).$$
Therefore $$|\det DF(s_{n},d_{n})|= |1+d_n (n_{s}(s_{n}) \cdot \tau (s_{n}))| |n(s_{n}) \cdot \tau^{\perp}(s_{n})| = \varphi^{\circ}(\nu(s_{n})) |1+d_n (n_{s}(s_{n}) \cdot \tau (s_{n}))|.$$

This shows that as long as $|d|$ is smaller than $1/C$, the determinant cannot degenerate. 
Moreover, since $n(s_{n}) \to n(s)$, we have  that convex combinations of the above limits are still matrices with full rank; hence 
 the implicit function theorem applies. 
It follows that $F$ is locally bi-Lipschitz and 
$F(\cdot,d)$ is $\varphi$-regular for every $|d|<1/C$.

Notice also that $\varphi(F(s,d)-u(s))=d$. To conclude that 
$u$ satisfies locally the $RW_\varphi$-condition with $R=1/C$,
it is enough to observe that $\varphi(F(s,t)-u(s'))\ge d$ for $s'$ in a neighborhood of $s$.
This is proven \cite[Lemma 3.4]{bellettini01} under the assumption that $u$ is an embedding,
but it also applies to our case as we can replace $u$ with a curve $\hat u$ such that
$\hat u$ is a
$\varphi$-regular embedding that coincides with $u$ in a neighborhood of $s$, with
$\| \hat n_{s} \|_{\infty} \leq C$.

Finally observe that if $n(s)=N(s)$ then $N(\cdot,d)$ is by construction the Cahn-Hoffman vector field that minimizes $\mathcal{K}$ as in Definition~\ref{def3}. (In principle this must be only verified on flat segments of $u$ where the curvature is not zero: here $N(\cdot)$ linearly interpolates between the values of $N$ at the end of the segment (as discussed for instance \cite[Ex. 7.1]{Be04}) and the same holds for $N(\cdot, d)$ by definition.)
\end{proof}

\begin{rem}\label{remd} \rm
Notice that, letting $(s(x),d(x))$ be a local inverse of $F$ in a neighborhood $U$ of $u(\bar s)$,
and letting $\widetilde n(x) = n(s(x))$, for $x\in U$
there holds 
\begin{equation}\label{eqd}
\widetilde n(x)\in \partial \varphi^\circ (\nabla d(x)),
\qquad \div \widetilde n(x) = \div n(s(x)) + O(d(x)).
\end{equation}
Notice that $d$ is a local $\varphi$-distance function and $s$ is a local projection function
on the support of the curve $u$ in the neighborhood $U$,
thus $\nabla d(x)= \frac{\nu(s(x))}{\varphi^{\circ}(\nu(s(x)))}$ and $\varphi^{\circ}(\nabla d(x)) =1$.
\end{rem}


In the case of smooth and elliptic anisotropies a stronger statement holds.
\begin{lem}\label{lemmadue}
Let $u$ be a closed curve which is locally a Lipschitz graph.
 Let $\varphi$ (and $\varphi^{\circ}$) be smooth and elliptic. The curve $u$ is $\varphi$-regular if and only if it locally satisfies the $RW_\varphi$-condition for some $R>0$.
 Moreover we have $R=\frac 1{\Vert N_s \Vert_\infty}$.
\end{lem}
\begin{proof}
One  part of the statement follows from Lemma~\ref{lemmauno}.
On the other hand, suppose now that $u$  satisfies locally the $RW_{\varphi}$-condition. Then (thanks to the regularity and ellipticity of $\varphi$), the curve is $\mathcal C^{1,1}$ which means that the Euclidian normal vector $\nu$ is Lipschitz with respect to the arc-length $s$. Then, since $\varphi^\circ$ is regular, the vector field
$$ N := D \varphi^\circ(\nu)$$
is also Lipschitz with respect to $s$ and is therefore the Cahn-Hoffman vector field.
\end{proof}

\subsection{$\varphi$-regular flows}
\begin{defi}\label{defflow}
We say that a family of curves $u:S^1\times [0,T]\to \R^2$ is a
$\varphi$-regular flow if

\begin{itemize}
\item $u$ is Lipschitz continuous,
\item $u(\cdot,t)$ is $\varphi$-regular for every $t\in [0,T]$, and there exists a family of Cahn-Hoffman vector fields $n(\cdot,t)$ 
for $u(\cdot,t)$ such that $\|n_{s} (\cdot, t)\|_\infty\le C$ for every 
$t\in [0,T]$,
\item for almost every $(x,t)$ there holds
\begin{align}\label{eqn}
u_{t}\cdot\nu = -(\div n)\, n\cdot\nu\,.
\end{align}
\end{itemize}

\end{defi}
\noindent It is easy to check that the Wulff shape shrinks self-similarly under \eqref{eqn}
(see \cite{gage94,St}).

\begin{rem}\label{remdflow}
Given a $\varphi$-regular flow $u$
we can define the functions $F(s,d,t),n(s,d,t)$,
as in Lemma \ref{lemmauno}.
We let $(s(x,t),d(x,t))$ be the inverse of $F$ in a (space-time) neighborhood $V$ of $u(\bar s,\bar t)$,
and we let $\widetilde n(x,t)=n(s(x,t),t)\in \partial \varphi^\circ(\nabla d)$. 
By \eqref{eqd} for $(x,t)\in V$ we get 
$$
\div \widetilde n(x,t) = \div n(s(x,t),t) + O(d(x,t))\,,
$$
moreover, using the equality $x=u(s(x,t),t) + d(x,t) n(s(x,t),t)$ and assuming that the curve $u(S^{1}, t)$ moves in direction $n$, one can show  
$$
d_t(x,t) = -u_t(s(x,t),t)\cdot \frac{\nu(s(x,t),t)}{\varphi^\circ(\nu(s(x,t),t))}\,.
$$ 
As a consequence,
in analogy to  \cite[Definition 2]{chambolle13}), \eqref{eqn}
is equivalent to 
\begin{align}\label{eqnd}
d_{t} = \div \widetilde n + O(d)
\qquad a.e.\,in\, V.
\end{align}
\end{rem}

When $u$ is smooth and the anisotropy $\varphi$ is smooth and elliptic,
the classical formulation of the anisotropic curvature flow (see \cite{angenent90}) 
is given by the equation 
\begin{align}\label{acsf}
u_t= \varphi^\circ(\nu) \kappa_\varphi \nu.
\end{align}
Notice that a solution $u$ of   \eqref{acsf} also satisfies \eqref{eqn}.
By setting
\begin{align}
  \phi (\theta):=\varphi^\circ(\nu)=\varphi^\circ(\cos \theta, \sin \theta),
\end{align}
a straightforward calculation gives
\begin{align}
\label{a1}
\phi(\theta)+ \phi''(\theta)=D^2 \varphi^\circ(\nu) \tau \cdot \tau ,
\end{align}
so that we can rewrite the  flow \eqref{acsf} as
\begin{align}
\label{ACSFF}
u_t=\phi(\theta)(\phi(\theta)+ \phi''(\theta)) \,\kappa \nu= \psi (\theta) \kappa \nu ,
\end{align}
where $\kappa$ is the Euclidean curvature and
\begin{align} \label{defpsi}
\psi(\theta):=\phi(\theta)(\phi(\theta)+ \phi''(\theta)) .
\end{align}
Note that by \eqref{ell-cond2}, the ellipticity of $\varphi$ implies that 
 $\psi \geq \tilde{C}>0$.

\section{Approximation of $\varphi$-regular curves}
\label{sec3}

In \cite[Lemma~3]{BCCN} (see also \cite[Remark~3]{chambolle13}) it is shown that given an anisotropy $\varphi$, it is possible to construct a sequence of anisotropies $\varphi_{\eps}$, such that  $\varphi_{\eps}$ and $\varphi_{\eps}^{\circ}$ are elliptic and belong to $C^{\infty}(\R^{2} \setminus \{0 \})$, and such that $\varphi_{\eps}$ converges uniformly to $\varphi$ on compact subsets of $\R^{2}$ as $\eps \to 0$. 

In this section we approximate a $\varphi$-regular curve $u$ by a family $u_{\eps}$
of smooth $\varphi_{\eps}$-regular curves, 
where $\varphi_{\eps} \gs \varphi$ is a family of $C^{\infty}(\R^{2} \setminus \{0 \})$   elliptic anisotropies approximating $\varphi$.

\begin{lem}\label{lem1}
Let $u$ be a $\varphi$-regular curve with $|\kappa_{\varphi}| \leq C$. Then, for a given sequence  $\varphi_\eps \to \varphi$ with $\varphi_\eps \in C^{\infty}(\R^{2} \setminus \{0 \})$,  elliptic and  $\varphi \leq \varphi_\eps$ (which is equivalent to $W_{\varphi_{\eps}} \subset  W_{\varphi}$), 
there exists a sequence of 
 curves $u_{\eps}$ of class $C^\infty$, such that $u_{\eps} \to u$ uniformly as $\eps \to 0$ and  $|\kappa_{\varphi_\eps}| \ls C'$ uniformly in $\eps$, where $C' >C$ is arbitrarily close to $C$.
\end{lem}

\begin{proof}
The main idea is to localize \cite[Lemma 1]{chambolle13}, where the result is proved for an embedded curve $u$ which is the boundary of a set $E$. The argument employed in \cite[Lemma 1]{chambolle13} is to take the union of all the Wulff shapes of type $x+RW_{\varphi_{\eps}}$, $R=1/C$, contained in the set $E$ (think for instance of  the situation where $E$ is a "thick"  $\sqcup$-shaped set and $\varphi$ is the $L^\infty$-norm). 
The boundary of this new set $\tilde{E}_{\eps}$ essentially regularizes the curve $u$ (``from inside'') attaching arcs of Wulff shapes  $RW_{\varphi_{\eps}}$ where $\partial E$ has  corners. 
Repeating the same  operation on the complementary $\tilde{E}_{\eps}^{c}$ one smoothes out also the remaining corners.
We shall adapt this method to the case of a general immersed curve, by exploiting the fact that the curve is locally a Lipschitz graph. Roughly speaking, the idea is to consider portions   $u([x-\delta, x+\delta])$ of the curve $u$, extend them with straight lines in order to obtain 
a global graph, 
and then perform the construction mentioned above.

\noindent{\it Step 1: covering $u(S^1)$ with graphs.}
Note that the curve $u$ is locally a Lipschitz graph. We cover $u(S^1)$ by $u([x_i - \alpha_i, x_i + \beta_i])$ such that
\begin{enumerate}[(i)]
 \item $ u(S^{1})= \cup_{i=1}^{N} u(I_{i})$ where $I_{i}=(x_{i}-\alpha_{i}, x_{i}+ \beta_{i}),$
 \item $u|_{[x_{i}-\alpha_{i}, x_{i}+ \beta_{i}]}$ is a graph and $u$ is differentiable at $x_i - \alpha_i$ and $x_i+\beta_i$,
 \item The intersection of the images of two consecutive graphs $u(I_{i}) \cap u(I_{i+1})$ contains a neighborhood of a point $y_i = u(w_i)$ such that $u$ and  the Cahn-Hoffman vector field $N$ are differentiable at $w_i$ and $\partial_s N(w_i) \neq 0$. In addition, if $y_i$ belongs to (the closure of) a straight line, we require that this line is entirely included in $u(I_{i}) \cap u(I_{i+1})$.
\end{enumerate}
This is possible since $N$ is Lipschitz and therefore differentiable almost everywhere and since every region of zero anisotropic curvature can be included in a graph. Indeed, in such regions, the Cahn-Hoffman vector field $N$ is constant. Therefore, since we have $\frac{\nu(x) }{\varphi^\circ (\nu (x))} \in \partial \varphi (N(x))$ (where $\nu(x)$ denotes the locally oriented  Euclidean normal to $u$ at $x$) and since at a point $N_{const}$ the subdifferential $\partial \varphi (N_{const})$ comprises at most a segment of finite length (recall that the duality map maps a unit ball into its dual ball), the normal $\nu$ is forced to remain in a non-flat cone and the region to remain a graph.

\noindent{\it Step 2: smoothing construction for graphs.} We prolongate each $u(I_i)$ by attaching two half lines with slope $\partial_x u(x_i-\alpha_i)$ and $\partial_x u(x_i + \beta_i)$. We then obtain the graph $\Gamma_f$ of a Lipschitz function $f$ with $u(I_{i}) \subset \Gamma_{f}$.
From our assumptions on $u$ and Lemma~\ref{lemmauno}, it follows that $\Gamma_{f}$ satisfies the $RW_\varphi$-condition (indeed, if this were not the case we get a contradiction by using \cite[Lemma 8.2]{bellettini012}).  Therefore also the $R'W_\varphi$-condition is satisfied for any $0< R' < R$. In the next step, we will choose a suitable $R' < R$ needed to glue together our local constructions. We now apply the result in \cite[Lemma 1]{chambolle13} to the subgraph $S_{f}$ of $f$, with $R'$ instead of $R$, and obtain a regularized set 
$$S_{x_{i},\eps} := \bigcup \left \{ (p+R'W_{\varphi_\eps}) \; \middle | \; (p + R'W_{\varphi_{\eps}}) \subset S_{f} \right\}.$$
Moreover  by construction, $S_{x_{i},\eps}$ satisfies the inner $R' W_{\varphi_\eps}$-condition and the outer  $R' W_{\varphi}$-condition.\\
Recalling that $\varphi_{\eps}\ge \varphi$, we have that 
the set $\partial S_{x_{i},\eps}\setminus \Gamma_{f}$ is a union of arcs of
 Wulff shape $R'W_{\varphi_{\eps}}$.
 We now define the curve $u_{\eps}^{-}$ by replacing $u(I_{i})
\setminus \partial S_{x_{i},\eps}$ by these arcs of Wulff shape.
 We show next  
that  the construction is compatible in some subregion of  $u(I_{i})\cap u(I_{i+1})$ so that the approximation curves $u_{\eps}^{-}$ of two adjacent local graphs can be  well connected/glued across this subregion.

\noindent{\it Step 3: connecting the approximate graphs.}
Let us now prove the main claim of Lemma \ref{lem1}, that is the constructions of $S_{x_i,\eps}$ and $S_{x_{i+1},\eps}$ are compatible: they can be connected in a canonical way. To do so, we will take advantage of the point $y_i$.

First, notice that the points $y \in u(S^1)$ where $u$ and $N$ are differentiable and $\partial_{s} N \neq 0$ are of two kinds. Either $y$ belongs to a flat edge of $u(S^1)$ around which $u$ is locally convex, which is parallel to an edge of the Wulff shape $RW_\varphi$ and whose length is bigger than the corresponding edge of $RW_\varphi$ (recall for instance Remark~\ref{rem3}), or $u$ is not flat around $y$ and the Wulff shape $p+RW_\varphi \subset S_{f}$ which touches $\Gamma_f$ at $y$ from inside  is not flat either around $y$.

We now prove that for $\eps$ small enough, $S_{x_{i},\eps}$ and $S_{x_{i+1},\eps}$ coincide around $y_i$, and therefore can be connected. We first deal with the case $y_i$ belongs to a flat edge $\ell$ of $u(S^1)$. Thanks to our assumptions $u(I_i) \cap u(I_{i+1})$ contains the whole $\ell$. The smoothing procedure leaves a part of this line unchanged (because of the reduction of $R$ into $R'$) when building $S_{x_{i},\eps}$ as well as $S_{x_{i+1},\eps}$, and the two curves can therefore be connected (see Figure \ref{fig:Wcrys}).

\begin{figure}
\begin{center}
\includegraphics[width=0.7\textwidth]{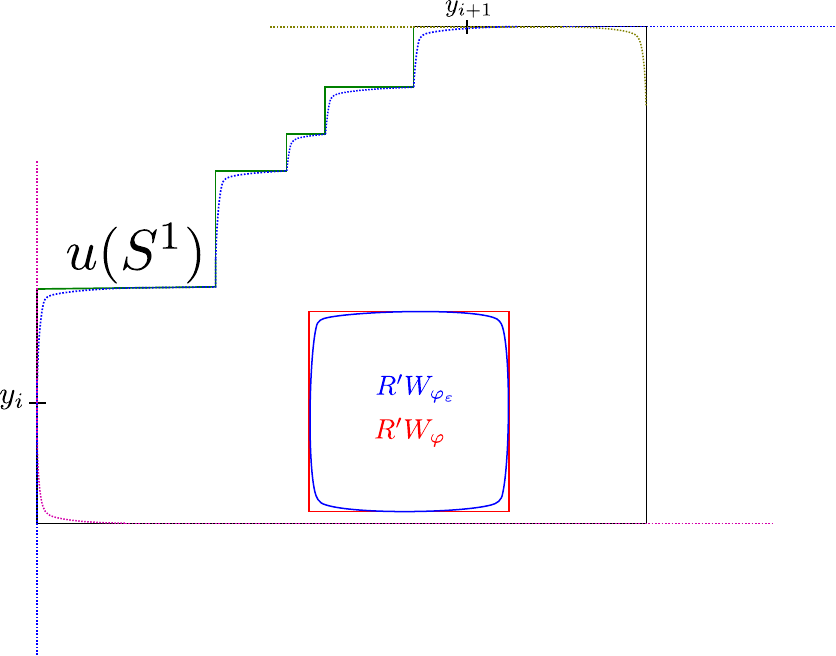}
\end{center}
\caption{Typical situation when $y_i$ belongs to a flat edge of $u(S^1)$.}
\label{fig:Wcrys}
\end{figure}

\begin{figure}
\begin{center}
\includegraphics[width=0.6\textwidth]{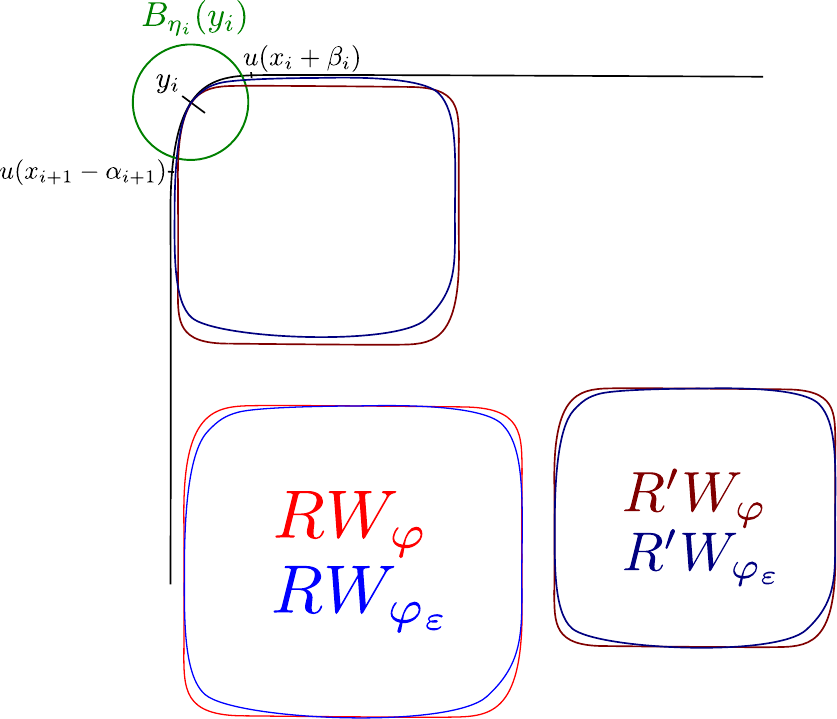}
\end{center}
\caption{Typical situation when $y_i$ belongs to a non-flat part of $u(S^1)$.}
\label{fig:Wsmooth} 
\end{figure}

Now, let us assume that $y_i$ belongs to a non-flat part of $u(S^1)$ (see Figure \ref{fig:Wsmooth}). Then, the curve $u$ at $y_i$ has a contact with a Wulff shape $p_{i }+RW_\varphi$, which cannot be flat around $y_i$ either. (In fact $u(S^{1})$ and $\partial (p_{i }+RW_\varphi)$ might even coincide in  a neighborhood of $y_{i}$; however the Wulff shape $\partial (p_{i }+RW_\varphi)$ is locally strictly convex at $y_{i}$.) As a result, the contact between $u(S^1)$ and the Wulff shape $p_{i}+R'W_\varphi$ takes place at $y_i$ only.
Let $\eta_{i}>0$ be such that $(B_{\eta_{i}}(y_i)\cap u(I_{i})) \subset u(I_i) \cap u(I_{i+1}) $.
 Since the approximate Wulff shape $R'W_{\varphi_\eps}$ Hausdorff converges to $R'W_\varphi$, it must touch $u(S^1)$ inside the ball of radius $\eta_{i}$, provided $\eps$ is small enough (depending on $\eta_{i}$, but these are in finite number anyway). Hence, every graph which coincides with $u(S^1)$ around $y_i$ is smoothed similarly on $B_{\eta_{i}}(y_i)$. We can therefore connect $S_{x_{i},\eps}$ and $S_{x_{i+1},\eps}$.

Eventually, we can connect all these pieces of graphs and obtain a curve $u_{\eps}^{-}$ which is locally a Lipschitz graph, and satisfies  locally  the inner $R' W_{\varphi_\eps}$-condition and the outer  $R' W_{\varphi}$-condition, where outer (resp. inner) means in the direction of the local orientation given by the vector $\nu$
(resp. $-\nu$). Notice also that, by construction, the Hausdorff distance between 
$u(I_{i})$ and $u_{\eps}^{-}(I_{i})$
is bounded above by the Hausdorff distance between $R' \partial W_{\varphi}$ and 
$R' \partial W_{\varphi_{\eps}}$, hence in particular it goes to zero as $\eps$ goes to zero.\\
On the other hand the ``added arcs'' $u(I_{i}) \setminus S_{x_{i}, \eps}$ which constitute parts of  $u_{\eps}^{-}$, locally satisfies the $R' W_{\varphi_{\eps}}$ condition, which is equivalent to say  that these arcs are $\varphi_{\eps}$-regular with $|\kappa_{\varphi_{\eps}}| \leq C'$, $C' =1/R'$ (see Lemma~\ref{lemmadue}). This means that the added arcs are of  class $C^{1,1}$ (and must admit a good local parametrization) and by construction concave (with respect to the local orientation). In particular their length is controlled above by the length of the replaced piece of curve and below by the distance of the endpoints of the arcs, and the length of $u_{\eps}^{-}$ tends to the length of $u$ as $\eps \to 0$.

\noindent{\it Step 4: building an actual curve.} 
We have explained so far how to construct geometrically  the approximation $u_{\eps}^{-}$. However we want some control of its parametrization too, since we claim uniform convergence of the approximate curves $u_{\eps}$ towards $u$. This can  be done by exploiting Dini's theorem on each local graph $u(\bar{I}_{i})$ and using the fact that we have a finite number of them.

We proceed similarly with $u^-_\eps$ instead of $u$ and by smoothing on the other side of the curve to build our approximation $u_\eps$, i.e we perform the analogous construction on the curve $u_{\eps}^{-}$ (which has the same orientation of $u$)
 by considering  the envelope of the Wulff shapes $R' W_{\varphi_{\eps}}$ that are locally \emph{above} $u_{\eps}^{-}$.

Finally, by a convolution argument we can assume that the curves $u_{\eps}$ 
are of class $C^{\infty}$ and satisfy  $|\kappa_{\varphi_{\eps}}| \leq C'$, with $C'> C$ arbitrarily close to $C$.
\end{proof}

\begin{rem}\label{remlength}\rm
Note that by construction the lengths of the curves $u_{\eps}(S^{1})$ are uniformly bounded from above and below.
\end{rem}


\section{Existence of $\varphi$-regular flows}

\subsection{Smooth and elliptic anisotropies}\label{asecsmooth}
Goal of this section is to show that if we have a bound of type $|\kappa_{\varphi_{\eps}}(0)| \leq C'$ for  initial smooth curves $u_{\eps}(0)$, then we can find a time interval (independent of $\eps$) where the anisotropic curve shortening flows \eqref{acsf} associated to the anisotropies $\varphi_{\eps}$ and initial data $u_{\eps}(0)$ exist.

In this section, for simplicity of notation  
we drop the index $\eps$ and write  $\varphi$ instead of  $\varphi_\eps$. In other words we assume $\varphi \in C^{\infty}(\R^{2} \setminus \{0 \})$  to be an elliptic anisotropy. 

First of all we  collect some important properties of the  anisotropic curvature flow \eqref{acsf} (or equivalently \eqref{ACSFF}).  For analogous results in the isotropic case see for instance \cite{CNV}.

In what follows it is important to keep in mind  that \eqref{ell-cond2} holds, hence $\psi \geq \tilde{C}$ (recall~\eqref{defpsi}). 

\begin{theo}{\cite[Theorem 3.1]{angenent90}}
There exists a smooth solution to \eqref{acsf} on $[0,T).$
\end{theo}

We start by deriving the evolution laws of  relevant geometric quantities.
\begin{lem}
\label{lemma2.1}
The following holds
\begin{align}\nonumber
\partial_t \partial_s (\cdot) &= \partial_s \partial_t (\cdot) +\psi(\theta) \kappa^2 \partial_s (\cdot)
\\\nonumber
\tau_t &= (\psi(\theta) \kappa)_s \nu
\\\nonumber
\nu_t & = -(\psi(\theta) \kappa)_s  \tau 
\\
\label{k_t}
\kappa_t &= (\psi(\theta) \kappa)_{ss} + \psi(\theta) \kappa^3 
\\\nonumber
\theta_t &=(\psi(\theta) \kappa)_s . 
\end{align}
\end{lem}
\begin{proof}
Let $f : S^1 \to \R^2$, $f=f(x,t)$. Then, we compute (note that the derivatives in $x$ and $t$ can commute)
\begin{align*}
\partial_t \partial_s f &= \partial_t \left( \frac{\partial_x f}{|u_x|}  \right) = - \frac{\scal{u_{tx}}{u_x}}{|u_x|^3} \partial_x f + \frac{\partial_t \partial_x f}{|u_x|} \\
&= - \scal{\partial_s(\psi(\theta) \k \nu)}{\tau} \partial_s f + \frac{\partial_x}{|u_x|} \partial_t f = \psi(\theta)  \k^2 \partial_s f + \partial_s \partial_t f.
\end{align*}  
Applying this formula to the other quantities, we get
$$\tau_t = \partial_t (\partial_s u) = \partial_s (\partial_t u) + \psi(\theta) \k^2 \partial_s u = \partial_s(\psi(\theta) \k \nu) + \psi(\theta) \k^2 \tau = (\psi(\theta) \k)_s \nu.$$
We prove the third formula similarly. \\
Writing $\k = \scal{\tau_s}{\nu}$, we get
$$\k_t = \scal{\partial_t \tau_s}{\nu} + \scal{\tau_s}{\partial_t \nu}.$$
Since $\tau_s$ is proportional to $\nu$ and $\nu_t$ is proportional to $\tau$, the second term vanishes. We obtain
$$ \k_t = \scal{\partial_s ((\psi(\theta) \kappa)_s \nu) + \psi(\theta) \k^2 \tau_s}{\nu} = (\psi(\theta) \k)_{ss} + \psi(\theta) \k^3.$$
Finally, recalling that $\nu = (\cos \theta,\sin\theta)$, one has $\nu_t = -\theta_t (\sin \theta, -\cos \theta)$, which implies the last formula.
\end{proof}

We now show that we can find a time interval  where the anisotropic curvature does not blow up, provided we know that its values at the initial time are bounded by a constant $C^{'}$. This result holds for \emph{any} smooth elliptic anisotropy $\varphi$ whose associated anisotropic curvature $\k_\varphi$ is bounded by $C'$ on the initial curve.

\begin{prop}\label{propK}
Suppose $\|\kappa_{\varphi}(0)\|_{L^{\infty}(S^{1})} \leq C'$. Then the anisotropic curvature~$\kappa_{\varphi}$ remains bounded on any time interval $[0,T ]  \subset [0,\frac{1}{2C'^2})$ (provided the flow exists on this time interval). Precisely we have
 $$\| \kappa_{\varphi}(t) \|_{L^{\infty}(S^{1})} \leq \frac{C'}{\sqrt{1-2T\|\kappa_{\varphi}(0)\|_{L^{\infty}(S^{1})}^{2} }}. $$
\end{prop}
\begin{proof}
First of all recall that due to \eqref{a1}  we have $\kappa_\varphi = \kappa (\phi + \phi'').$ For simplicity we denote by $h$ the quantity $\phi + \phi''$. By Lemma~\ref{lemma2.1} we infer that
\begin{align*}
 \partial_t \kappa_\varphi &= \partial_t( \kappa h) = h[(\partial_{ss}(\psi(\theta) \k) + \kappa^3 \psi ]+ \k h'\partial_s(\psi(\theta) \k ) \\
& = h[\partial_s(\k^2 \psi' + \k_s \psi) + \k^3 \psi  ]+ \k h' (\k^2 \psi'+\k_s\psi) \\
&= h(3\k \k_s \psi' + \k^3 \psi'' + \k_{ss} \psi+ \k^3\psi) +\k^3 h' \psi' + \k \k_s h' \psi
\end{align*}  
and
$$\partial_{ss} \kg  = \partial_{ss} (\k h) = \partial_s(\k_s h + \k^2 h') = \k_{ss} h +3\k_s \k h' + \k^3 h''.$$
Noting that
$$(\partial_t - \psi \partial_{ss}) \frac{\kg^2}{2} = \kg \partial_t \kg - \psi \kg \partial_{ss} \kg - \psi (\partial_s \kg)^2$$
we get  
\begin{align*}
 (\partial_t - \psi \partial_{ss})& \frac{\kg^2}{2} = -\psi \kg \left[ \k_{ss} h + 3 \k_s \k h' + \k^3 h''\right] - \psi(\partial_s \kg)^2 \\
&\quad + \kg \left[ h(3\k \k_s \psi' + \k^3 \psi'' + \k_{ss} \psi + \k^3\psi) +\k^3 h' \psi' + \k \k_s h' \psi \right] \\
&= \k_s\kg \k (3h\psi' -2h'\psi) + \kg \k^3(h\psi + h\psi''+h'\psi' - \psi h'') - \psi(\partial_s \kg)^2.
\end{align*} 
Now, note that since
$$\psi' = (h\phi)' = h' \phi + h \phi',$$
$$\psi'' = h'' \phi + 2 h'\phi' + h \phi'',$$
we obtain
$$3h\psi' - 2 h'\psi = 3h^2\phi'+ h h' \phi$$
and 
$$h\psi+ h\psi''+h'\psi' - \psi h'' = h^2 \phi + hh''\phi +2hh'\phi' + h^2\phi'' + (h')^2 \phi +hh'\phi' - hh''\phi = h^3+3hh'\phi' + (h')^2\phi.$$
As a result, 
$$(\partial_t - \psi \partial_{ss}) \frac{\kg^2}{2} =  \k_s \kg \k(3h^2\phi'+ h h' \phi)+\kg \k^3( h^3+3hh'\phi' + (h')^2\phi) - \psi (\partial_s \kg)^2.$$  
Since
$$\partial_s \frac{\kg^2}{2} =\kg(\k_s h + \k^2 h'),$$
we can write 
$$\k_s \kg \k  (3h^2 \phi' +hh' \phi) + \k^3 \kg (3hh' \phi' + (h')^2 \phi) = (3 \k h \phi' + h' \k \phi)  \partial_s \frac{\k_\varphi^2}{2}$$ 
which yields
$$(\partial_t - \psi \partial_{ss}) \frac{\kg^2}{2} \ls (3 \k h \phi' + h' \k \phi  )  \partial_s \frac{\k_\varphi^2}{2}
 + \k_\varphi^4. $$

At a maximal point for $\kg^2$, the quantity $\partial_s \kg^2$ vanishes and $\partial_{ss} \kg^2$ is nonpositive. As a result, letting $g:=\max\limits_{S^1} \kg^2$ and by \cite[Lemma~2.1.3]{mantegazza}, we have $\frac{d}{dt} g \ls 2 g^2$, which implies

\begin{equation}g(t) \ls \frac{g(0)}{1-2tg(0) }\label{noblow}\end{equation} 
 as long as $1-2t g(0) >0$. Since by assumption $g(0) \ls (C')^2$, the anisotropic curvature $\kg$ cannot blow up on a time interval $[0,T ]  \subset [0,\frac{1}{2C'^2})$.
\end{proof}

The rest of this section is devoted to showing that if the maximal time of existence of the flow is finite, then both the isotropic and anisotropic curvatures have to blow-up.

For the evolution of the derivatives of the curvature we have
\begin{lem}
\label{lemma2.2}
For $j \in \mathbb{N}$, $j \geq 1$ we have
\begin{align}
\partial_t (\partial_s^j \kappa) &=\psi(\theta) (\partial_s^j \kappa)_{ss} + 
(j+3) \psi'(\theta) \kappa (\partial_s^j \kappa)_s  \notag\\
& \quad + P_{j}(\psi, \psi', \psi'', \kappa, \kappa_s) \partial_s^j \kappa + Q_{j}(\psi, \psi', \ldots, \psi^{(j+2)}, \kappa, \ldots, \partial_s^{j-1} \kappa)
\end{align}
where $P_{j}(\cdot)$ and $Q_{j}(\cdot)$ are polynomials in the given variables and $\psi^{(m)}=\partial_\theta^m \psi$.
\end{lem}
\begin{proof}
The proof is by induction on $j$ and relies on Lemma~\ref{lemma2.1} and  the fact that $\psi(\theta)_s =\psi'(\theta) \kappa$.
For $j=1$ we have that 
\begin{align} \label{aiutino}
\partial_{t} \partial_{s} \kappa  & = \partial_{s} \partial_{t} \kappa
 + \psi(\theta) \kappa^{2} \partial_{s} \kappa\\
&= \partial_{s} (  (\psi(\theta) \kappa)_{ss} + \psi (\theta) \kappa^{3}) + \psi(\theta)\kappa^{2} \partial_{s} \kappa  \notag \\
& = \psi(\theta) \kappa_{sss} + 4 \psi'(\theta) \kappa \kappa_{ss} + \psi^{(3)}(\theta) \kappa^{4} + \psi' (\theta) \kappa^{4} + (6 \psi''(\theta) \kappa^{2} +4\psi(\theta) \kappa^{2} +3 \psi'(\theta) \kappa_{s}) \kappa_{s}. \notag
 \end{align}
 The induction step follows with similar arguments.
\end{proof}

\begin{lem}
Let $w:=\log |u_x|$. There holds
\begin{align}
w_t =-\psi(\theta) k^2.
\end{align}
In particular $\| u_x (t) \|_\infty \leq \| u_x (0)\|_\infty$.
\label{lemlip}
\end{lem}
\begin{proof}
A direct computation gives
\begin{align*}
w_t= \tau \cdot \partial_s u_t =\tau \cdot \psi(\theta) \kappa \nu_s =-\psi(\theta) k^2.
\end{align*}
The second statement follows from $\psi \geq 0$.
\end{proof}
Note that if we have a bound on the curvature, then from $-w_t \leq C(\|\kappa\|_\infty, \|\psi\|_\infty)$ we also infer that $|u_x(t)| \geq (\inf_{S^1} |u_x(0)|) e^{-C(\|\kappa\|_\infty, \|\psi\|_\infty) t}$.

\begin{lem}
\label{lemma2.4}
Assume that \eqref{acsf} has a smooth solution on $[0, \bar{t}]$, with $\bar{t}>0$. Then
$$\max_{S^1 \times [0, \bar{t}]} |\partial_s^j \kappa| \leq C_j, \qquad (j \in \mathbb{N}),$$
where $C_j$ depends on $\bar{t}$, $\tilde{C}$ (as in \eqref{ell-cond2}), $\| \psi^{(l)} \|_\infty$ for $l=0, \ldots, j+2$, $C_l$ for $l\leq j-1$,  
$\|\partial^j_s \kappa (0)\|_{\infty}$, and $\max_{S^1 \times [0, \bar{t}]}|\kappa|$.
\end{lem}
\begin{proof} 
The proof goes by induction on $j$.
Let $v=\partial_s^j \kappa$. From Lemma \ref{lemma2.2} we know that
\begin{align*}
v_t =\psi(\theta) v_{ss} + (j+3) \psi'(\theta) \kappa v_s + P_{j}(\psi, \psi', \psi'', \kappa, \kappa_s) v + Q_{j}(\psi, \psi', \ldots, \psi^{(j+2)}, \kappa, \ldots, \partial_s^{j-1} \kappa)
\end{align*}
(where recall that  $\partial_s =\frac{1}{|u_x|} \partial_x$ and $ v_{ss}=\frac{1}{|u_x|^2}v_{xx} -   \frac{v_x}{|u_x|} \tau \cdot \frac{u_{xx}}{|u_x|^2} $). Together with $\psi (\theta) \geq \tilde{C}>0$ (the anisotropy is  elliptic) 
we obtain a parabolic quasilinear equation for which we can apply arguments given in \cite[Thm. 9.5]{lieberman05}. 

More precisely let us look at the case where $j \geq 2$. Without loss of generality we may assume that there exists a point in $  S^1 \times (0,\bar{t}]$ where $v$ attains a positive maximum (if not argue with $-v$). The map $v$ satisfies an equation of type
$$ 0=-v_t + \psi(\theta)v_{ss} + a(s,v,v_s)$$
where, in view of the induction hypothesis, we have that $a(s,v,0) \leq c(|v| +1) \leq \alpha |v| +\frac{\beta}{|v|} $ with positive  constants $\alpha$ and $\beta$ depending on $\| \psi^{(l)} \|_\infty$ for $l=0, \ldots, j+2$, $C_l$ for $l\leq j-1$, and $\max_{S^1 \times [0, \bar{t}]}|\kappa|$. Set $\lambda=-\alpha -1$. Suppose $P=(x,t) \in S^1 \times (0, \bar{t}]$ is a point in which
$m :=e^{\lambda t}v$ attains a positive maximum. Then $m_t = \lambda e^{\lambda t}v + e ^{\lambda t}v_t$ and at $P$ we have
$m_x=0$, $m_t \geq 0$ (thus $m_s=v_s=0$), $m_{xx}\leq 0$ (thus $m_{ss} \leq 0$, $v_{ss} \leq 0$). At $P$ (where $v >0$) we have
\begin{align*}
0 &= -v_t + \psi(\theta) v_{ss} + a(s,v,v_s) \leq-v_t + \alpha |v| + \frac{\beta}{|v|}
 = (\lambda +\alpha) v + \frac{\beta}{ v} = - v + \frac{\beta}{ v}.
\end{align*}
Thus $v(P) \leq \sqrt{\beta}$ and we infer that
$$ \sup_{S^1 \times [0, \bar{t}]} v\leq e^{(\alpha+1)\bar{t}}(\sqrt{\beta} + \sup_{S^1}v^+(0)).$$
Arguing in the same way with $-v$ instead of $v$,
we get a bound also on $v^-$ and therefore on $|v|$.

For the case $j=1$, the conclusion of Lemma \ref{lemma2.2} does not let us apply the same maximum principle (see the quadratic term appearing in \eqref{aiutino}). To cope with this difficulty, we will study the quantity $v:=\psi(\theta)^{\frac 32} \k$. To this aim, let us compute using Lemma~\ref{lemma2.1}
 (for simplicity we drop the dependence on $\theta$ in the formulas for $\psi$ and its derivatives) 
$$\partial_t (\psi^{3/2}(\theta) \k) = \frac 32 (\k^2 \psi' + \k_s \psi) \psi' \psi^{1/2}\k + \psi^{3/2} (3\k \k_s \psi' + \k^3 \psi'' + \k_{ss} \psi + \psi \k^3) $$
whereas
\begin{equation}\partial_s(\psi^{3/2} \k) = \frac 32 \k^2 \psi' \psi^{1/2} + \psi^{3/2} \k_s. \label{eq:ks}\end{equation}
Then, we have
\begin{align*}
 \partial_{ss} (\psi^{3/2} \k) &= 3 \k \k_s \psi' \psi^{1/2}+ \frac 32 \k^3 \psi'' \psi^{1/2} + \frac 34 \k^3 (\psi')^2 \psi^{-1/2} + \frac 32 \k \psi' \psi^{1/2} \k_s + \psi^{3/2} \k_{ss} \\
 &= \frac 92 \k \k_s \psi' \psi^{1/2} + \frac 32 \k^3 \psi'' \psi^{1/2} + \frac 34 \k^3 (\psi')^2 \psi^{-1/2} +\psi^{3/2} \k_{ss}
\end{align*}
and
\begin{align*}
 \partial_{sss} (\psi^{3/2} \k) &= \frac 92 \k_s^2 \psi' \psi^{1/2} + \frac 92 \k \k_{ss} \psi' \psi^{1/2} + \frac 92 \k^2 \k_s \psi'' \psi^{1/2} \\
 &+ \frac 94 \k^2 \k_s (\psi')^2 \psi^{-1/2} + \frac 92 \k^2 \k_s \psi'' \psi^{1/2} + \frac 32 \k^4 \psi''' \psi^{1/2} + \frac 34 \k^4 \psi'' \psi' \psi^{-1/2} \\
 &+ \frac 94 \k^2 \k_s (\psi')^2 \psi^{-1/2} + \frac 32 \k^4 \psi' \psi'' \psi^{-1/2} - \frac 38 \k^4 (\psi')^3 \psi^{-3/2} + \frac 32 \k \psi' \psi^{1/2} \k_{ss} + \psi^{3/2} k_{sss} \\
 &=\frac 92 \k_s^2 \psi' \psi^{1/2} + 6 \k \k_{ss} \psi' \psi^{1/2} + 9 \k^2 \k_s \psi'' \psi^{1/2} + \frac{9}{2} \k^2 \k_s (\psi')^2 \psi^{-1/2} \\ 
 &+ \frac 32 \k^4 \psi''' \psi^{1/2} + \frac 94 \k^4 \psi'' \psi' \psi^{-1/2} - \frac 38 \k^4 (\psi')^3 \psi^{-3/2}  + \psi^{3/2} k_{sss}.
\end{align*}
Recall Lemma~\ref{lemma2.1} and the calculations performed in Lemma~\ref{lemma2.2}, we also compute
\begin{align*}
 \partial_t(\partial_s(\psi^{3/2} \k)) &= 3 \k \k_t \psi' \psi^{1/2} + \frac 32 \k^2 (\psi' \k^2 + \k_s \psi) \psi'' \psi^{1/2} 
 \\ &+ \frac 34 \k^2 (\psi')^2 (\psi' \k^2 + \k_s \psi) \psi^{-1/2} + \frac{3}{2} (\psi' \k^2 + \k_s \psi) \psi' \psi^{1/2} \k_s + \psi^{3/2} \partial_t \k_s \\
 &=3 \k  \psi' \psi^{1/2} (3\k \k_s  \psi' + \k^3 \psi'' + \k_{ss} \psi + \psi \k^3)  \\
 &+ \frac 32 \k^2 (\psi' \k^2 + \k_s \psi) \psi'' \psi^{1/2}+ \frac 34 \k^2 (\psi')^2 (\psi' \k^2 + \k_s \psi) \psi^{-1/2} + \frac{3}{2} (\psi' \k^2 + \k_s \psi) \psi' \psi^{1/2} \k_s  \\
 &+\psi^{5/2} \kappa_{sss} + 4 \psi^{\frac{3}{2}} \psi' \kappa \kappa_{ss} + \psi^{\frac{3}{2}} \psi^{(3)} \kappa^{4} + \psi^{\frac{3}{2}}  \psi'  \kappa^{4} + \psi^{\frac{3}{2}} (6 \psi'' \kappa^{2} +4\psi \kappa^{2} +3 \psi'\kappa_{s}) \kappa_{s} \\
 &=9\k^2 \k_s  (\psi')^2 \psi^{1/2} + 3 \k^4 \psi' \psi'' \psi^{1/2} + 3 \k \k_{ss} \psi' \psi^{3/2} + 3\k^4 \psi' \psi^{3/2}   \\
 &+ \frac 32 \k^2 (\psi' \k^2 + \k_s \psi) \psi'' \psi^{1/2}+ \frac 34 \k^2 (\psi')^2 (\psi' \k^2 + \k_s \psi) \psi^{-1/2} + \frac{3}{2} (\psi' \k^2 + \k_s \psi) \psi' \psi^{1/2} \k_s  \\
  &+\psi^{5/2} \kappa_{sss} + 4 \psi^{3/2} \psi' \kappa \kappa_{ss} + \psi^{3/2} \psi^{(3)} \kappa^{4} + \psi^{3/2}  \psi'  \kappa^{4} + \psi^{3/2} (6 \psi'' \kappa^{2} +4\psi \kappa^{2} +3 \psi'\kappa_{s}) \kappa_{s} .
\end{align*}
As a result, we get
\begin{align*}
\partial_t(\partial_s(\psi^{3/2} \k))& = \psi \partial_{ss}(\partial_s(\psi^{3/2} \k))  +\k\psi' (\partial _s(\psi^{3/2} \k))_s \\
&+ \partial_s(\psi^{3/2} \k) \cdot P(\k,\psi,\psi^{-1},\psi', \psi^{''}) + Q(\k,\psi^{1/2},\psi^{-1/2},\psi^{3/2},\psi, \psi',\psi'',\psi^{(3)}),
\end{align*}
where $P$ and $Q$ are polynomials in the given variables.

Finally, as before, $v= \partial_s(\psi^{3/2} \k)$ satisfies an equation of the type
$$ v_t = \psi(\theta) v_{ss} + a(s,v,v_s),$$ where $a(s,v,0) \ls c(|v|+1)$ and we can conclude again that $v$ is bounded. This, thanks to \eqref{eq:ks} and the result for $j=0$ ($\k$ is bounded), implies that $\k_s$ is bounded and concludes the proof.
\end{proof}

\begin{prop}\label{protime}
Let $T$ be the maximal time of existence of \eqref{acsf} and assume that $T < \infty$. Then
\begin{align}\label{eqmax}
\limsup_{t \to T} \| \kappa\|_\infty =+\infty.
\end{align}
\end{prop}

\begin{proof}
Assume that $|\kappa|$ is uniformly bounded for all $ t \in [0, T)$. Then the previous lemmas imply a uniform bound on $|u_x|$, $|u_x|^{-1}$ and  $|\partial_s^j \kappa|$.
Using \eqref{acsf} we can write
$$ u(x,t_{2})-u(x,t_{1}) = \int_{t_{1}}^{t_{2}}\varphi^\circ(\nu(x)) (D^{2}\varphi^\circ(\nu(x)) \tau(x) \cdot \tau(x) ) \k(x) \nu(x)  dt.$$
The bounds on $\k$ and on the anisotropy map imply that $u(x, t_{2})$ has a limit when $t_{2} \to T$. It remains to show that the convergence of $u(\cdot,t)$ is in fact in $\mathcal C^\infty.$ 
This is achieved by showing that we can get uniform bounds (in time) for all derivatives of the map $u$ in the original parametrization.

First of all note that for a function $h : S^1 \to \R$ we have   that 
\begin{align}\label{eqH}
\partial_x^m h - |u_x|^m \partial_s^m h=P_{m}(|u_x|, \ldots, \partial_x^{m-1} |u_x|, h, \ldots, \partial_s^{m-1} h),
\end{align}
where $P_{m} $ is a polynomial in the given variables. Therefore if we can show that uniform bounds hold for the derivatives of the length element $|u_{x}|$ then using Lemma~\ref{lemma2.4} we obtain bounds for the derivatives $|\partial_{x}^{m} \kappa|$ and $|\partial^m_x u_x|$ on $(0,T)$.

It remains to show that the derivatives of the length element $z =|u_x|$  stay bounded.
Differentiating the PDE for $z =|u_x|$, namely
$$ z_t=-\psi(\theta)\kappa^2 z,$$
we can write
\begin{equation}\label{blabla}
 (\partial_x^m z)_t= -\psi(\theta) \kappa^{2}\partial_x^m z +\sum_{i+j=m, j \leq m-1} c(i,j,m) \partial_{x}^{i} (\psi(\theta) \kappa^{2}) \partial_{x}^{j}z  
 \end{equation}
for some coefficients $c(i,j,m)$. Here we proceed by induction. Assuming that $|\partial_x^j z|$ is bounded up to order $m-1$, then equation \eqref{eqH} applied to $h=\psi(\theta) \kappa^{2} $ and Lemma~\ref{lemma2.4} give boundedness of the terms appearing in \eqref{blabla}, so that we infer
$$ (\partial_x^m z)_t \leq  -\psi(\theta) \kappa^{2} \partial_x^m z + c.$$
 A Gronwall argument yields then boundedness of $\|\partial_x^m z\|_{\infty}$ on $(0,T)$.

Having achieved  $\mathcal C^\infty$ convergence, we can now extend  $u$ past $T$, which gives a contradiction. The claim follows.
\end{proof}

Note that since $\kappa_{\varphi}=D^{2}\varphi^\circ(\nu) \tau \cdot \tau \kappa$ the previous proposition implies that also the anisotropic curvature blows-up if the maximal time of existence of the flow \eqref{acsf} is finite.

Similarly to the isotropic case we get a lower bound for the curvature as follows.

\begin{lem}
Let $T$ be the maximal time of existence of \eqref{acsf} and suppose $T < \infty$. Then
\begin{align}
\liminf_{t \to T} \sqrt{T-t} \| \kappa \|_{L^\infty} \geq \frac{1}{\sqrt{2 \alpha}}
\end{align}
where $\alpha= \max_{S^1} |\psi +\psi''|$.
\end{lem}
\begin{proof}
Let $w:= \kappa^2$. Then from \eqref{k_t} we infer that
\begin{align*}
w_t & =\psi(\theta) w_{ss} +2(\psi(\theta) + \psi''(\theta))w^2 +
3\psi'(\theta)w_s \sqrt{w} \, \mbox{sign}(\kappa)- 2\psi(\theta)(k_s)^2\\
& \leq \psi(\theta) w_{ss} +2(\psi(\theta) + \psi''(\theta))w^2 +
3\psi'(\theta)w_s \sqrt{w}  \, \mbox{sign}(\kappa).
\end{align*}
Let $M(t):= \max_{S^1} w \geq 0$. Then using \cite[Lemma~2.1.3]{mantegazza} we infer
\begin{align*}
\frac{d}{dt} M(t) \leq 2|\psi + \psi'' | M^2(t) \leq 2 \alpha M^2(t) \leq 2 \alpha (M(t)+ \delta)^{2},
\end{align*}
where $\alpha= \max_{S^1} |\psi +\psi''|$ and $\delta>0$. Integrating on $[t,s] \subset [0,T)$ we obtain
$$ -\frac{1}{M(s)+\delta}+ \frac{1}{M(t)+\delta} \leq 2 \alpha (s -t).$$
Letting $s \to T$ along a sequence of times so that $M(s) \to \infty$ by the Proposition \ref{protime}, and  choosing $\delta$ arbitrary small we get the claim.
\end{proof}

\subsection{General anisotropies}\label{secgen}

We now show short time existence of $\varphi$-regular flows,
starting from a $\varphi$-regular initial curve.

\begin{theo}\label{limitnorm}
Let $u_0$ be a closed $\varphi$-regular curve. Then there exist $T>0$ and a 
$\varphi$-regular flow $u$ on $[0,T]$, with $u(0)=u_0$.
\end{theo}

\begin{proof}
Let $\varphi_\eps \gs \varphi$ be a family of smooth and elliptic 
anisotropies converging to $\varphi$ as $\eps\to 0$.
Let $u_{\eps,0}$ be the approximations of $u_0$ described in Lemma~\ref{lem1}. 
Recalling Remark \ref{remlength},
up to a reparametrization, we can assume that $|(u_{\eps,0})_x|\le C$, so that the curves $u_{\eps,0}$ are equi-Lipschitz in $\eps$.

Denote by $u_{\eps}$ the solutions to \eqref{eqn} with initial data $u_{\eps,0}$,
that is, the functions $u_{\eps}$ solve the equation
\begin{align}\label{eqneps}
(u_{\eps })_{t}=\kappa_{\varphi_{\eps}} N_{\eps},
\end{align}
where $N_\eps = D\varphi^\circ_\eps(\nu_\eps)$ is the Cahn-Hoffman vector field.

In view of Lemma~\ref{lem1} and Proposition~\ref{propK}
the curves ${u}_{\eps}$ have 
curvature $\kappa_{\varphi_\eps}$  bounded by a constant $\Lambda$ which does not
depend on $\eps$ and $t$, as long as $t$ does not reach a certain $T$ which depends only on $u_0$. As a consequence, by Proposition \ref{protime}
we can assume that the solutions $u_\eps$ are all defined in the same time interval
$[0,T]$.
Our goal is to pass to the limit in $u_\eps$ as $\eps\to 0$. 
 
Thanks to Lemma \ref{lemlip} the maps $u_\eps$
are equi-Lipschitz in space. 
Moreover, by Proposition~\ref{propK} and \eqref{eqneps} the solutions $u_\eps$
are also equi-Lipschitz in time. Using Ascoli-Arzela Theorem  we can ensure that $u_\eps$ converge, up to a subsequence, to some Lipschitz function $u$ such that $u(0)=u_0$. We claim that $u$ is a $\varphi$-regular flow.

Fix $(\bar x, \bar t)\in S^1\times (0,T)$. By Remark \ref{remdflow}, in a neighborhood $V$
of 
$(u(\bar x,\bar t), \bar t)$ with 
$u(\bar x,\bar t)=\lim_\eps u_\eps(\bar x,\bar t)$ equation \eqref{eqneps} can be rewritten as
\begin{align}\label{eqndeps}
\widetilde N_\eps= D \varphi^\circ_{\eps}(\nabla d_\eps), \qquad
(d_\eps)_{t} = {\rm div} \widetilde N_\eps + O(d_\eps)
\qquad \text{a.e. in $V$.}
\end{align} 
where the functions $d_\eps,\,\widetilde N_\eps$ are defined 
as in Remarks \ref{remd} and \ref{remdflow}.
In particular, $d_\eps$ is the $\varphi_\eps$-distance function from the support of $u_\eps$,
restricted to a neighborhood of $(u(\bar x,\bar t), \bar t)$
 and $\widetilde N_\eps$ is a suitable extension of the Cahn-Hoffman field $N_\eps$. 

From the convergence of $u_\eps$ to $u$ we immediately get the uniform convergence 
of $d_\eps$ to some function $d$ in $V$. 
Notice that $d$ is the $\varphi$-distance function to the support of $u$, restricted to a neighborhood
of  $(u(\bar x,\bar t), \bar t)$. 
We now show that the fields $\widetilde N_\eps$ also 
converge to some field $\widetilde n$ which is the extension of 
a Cahn-Hoffman vector field $n$ for $u$.
By  the fact that the curvatures $\kappa_{\varphi_\eps}$
are uniformly bounded and the control on the lengths from above and below (Remark~\ref{remlength}, Lemma~\ref{lemlip} and comments below), it follows for every fixed time $t$ that $N_{\eps}(\cdot, t)$ are equi-Lipschitz with respect to arc-length and (up to a subsequence) converge uniformly to a Lipschitz field $n(s,t)$. Let $\widetilde n$ be the extension in $V$ (cf.  Remark~\ref{remdflow})  satisfying  $ n(s(x), t)=\widetilde n(s(x), t)= \widetilde n(x,t)$.
Since $\widetilde N_\eps$ are uniformly bounded in $V$, they converge, up to a subsequence in the weak* topology of $L^\infty(V)$, to the vector field $\widetilde n\in L^\infty(V)$. 
By $\varphi_{\eps} \to \varphi$ we infer that $\varphi(\widetilde n)\le 1$.
 Again using the uniform boundedness of the anisotropic curvatures $\kappa_{\varphi_\eps}$ and  \eqref{eqd} we infer that  $(\div \widetilde N_{\eps})$ weak*- converges to $(\div   \widetilde n)$ in $L^\infty(V)$.  
We can now pass to the limit in \eqref{eqndeps} and obtain that $d$ satisfies
\begin{align}\label{eqndd}
d_{t} = {\rm div} \widetilde N + O(d)
\qquad \text{a.e. in $V$.}
\end{align}
We now show that $\widetilde N$ satisfies the inclusion 
\begin{align}\label{incN}
\widetilde N\in \partial \varphi^o(\nabla d)\qquad \text{a.e. in $V$.}
\end{align}
Indeed, letting $\psi\in C^1_c(V)$ and recalling that 
$\widetilde N_\eps\cdot \nabla d_\eps = \varphi^o_\eps(\nabla d_\eps)=1$, we have
$$
\int_V \psi \,dx dt = \int_V \psi\, \widetilde N_\eps\cdot \nabla d_\eps\, dx dt
= - \int_V d_\eps \big( \widetilde N_\eps\cdot \nabla \psi + \psi {\rm div} \widetilde N_\eps
\big) dx dt.
$$
Passing to the limit in the right-hand side, we get
$$
\int_V \psi \, dx dt = -  \int_V d \big( \widetilde N\cdot \nabla \psi 
+ \psi {\rm div} \widetilde N \big) \,dx dt = 
\int_V \psi\, \widetilde N\cdot \nabla d\, dx dt,
$$
which is equivalent to \eqref{incN}.

It then follows that $u$ is a $\varphi$-regular flow on $[0,T]$,
which proves the thesis.
\end{proof}

The uniqueness of $\varphi$-flows for general anisotropies and initial 
data is still an open problem. However,
it has been proved in \cite{BeNo:99} that 
the evolution is unique
if the initial curve is embedded. 
Moreover, in the purely crystalline case, that is when $W_\varphi$ is a polygon
and when the initial curve is piecewise linear and $\varphi$-regular, 
the problem reduces to a family ODE's and the solution is therefore unique 
(see \cite{GiGu:96}).

\end{document}